\documentclass{amsart}

\usepackage{latexsym,amssymb,amsmath,amsopn,graphics,xy,epsfig,picture,epic}

\usepackage{color}

\def\cc{{\mathcal C}}

\def\ff{{\mathcal F}}

\def\ss{{\mathcal S}}

\def\ffi{\varphi}
\def\eps{\varepsilon}
\def\dst{\displaystyle}

\DeclareMathOperator{\supp}{supp}
\DeclareMathOperator{\diam}{diam}

\def\H{{\mathcal{H}}}

\def\N{{\mathbb{N}}}

\def\R{{\mathbb{R}}}
\def\S{{\mathbb{S}}}

\def\Z{{\mathbb{Z}}}

\newcommand{\norm}[1]{{\left\|{#1}\right\|}}
\newcommand{\ent}[1]{{\left[{#1}\right]}}
\newcommand{\abs}[1]{{\left|{#1}\right|}}
\newcommand{\scal}[1]{{\left\langle{#1}\right\rangle}}

\newenvironment{remark}[1][]{\vskip3pt\noindent\textbf{Remark.}\ }{\rm\vskip3pt}

\newenvironment{fact}[1][{\sl}]{\vskip3pt\noindent\textbf{Theorem}\ }{\sl\vskip3pt}

\newtheorem{lemma}{Lemma}[section]
\newtheorem{proposition}[lemma]{Proposition}
\newtheorem{theorem}[lemma]{Theorem}
\newtheorem{corollary}[lemma]{Corollary}

\newtheorem{examplenum}[lemma]{Example}

\date{\today}

\begin{document}

\title[Uncertainty Principle in Besov spaces]{An uncertainty principle and sampling inequalities in Besov spaces}

\author[Ph. Jaming]{Philippe Jaming}
\address{Institut de Math\'ematiques de Bordeaux UMR 5251,
Universit\'e de Bordeaux, cours de la Lib\'eration, F 33405 Talence cedex, France}
\email{philippe.jaming@math.u-bordeaux1.fr}

\author[E. Malinnikova]{Eugenia Malinnikova}
\address{Department of Mathematical Sciences,
Norwegian University of Science and Technology (NTNU)
7491 Trondheim, Norway}
\email{eugenia@math.ntnu.no}


\thanks{
Ph.J. kindly acknowledge financial support from the French ANR programs ANR
2011 BS01 007 01 (GeMeCod), ANR-12-BS01-0001 (Aventures).
This study has been carried out with financial support from the French State, managed
by the French National Research Agency (ANR) in the frame of the ”Investments for
the future” Programme IdEx Bordeaux - CPU (ANR-10-IDEX-03-02).\\
E.M. was partly supported by Project 213638 of the Research Council of Norway.\\
This research was sponsored by the French-Norwegian PHC AURORA 2014 PROJECT N° 31887TC, N~233838, {\it CHARGE}}

\begin{abstract}
We extend Strichartz's uncertainty principle \cite{str}
 from the setting of the Sobolev space
$W^{1,2}(\R)$ to more general Besov spaces $B^{1/p}_{p,1}(\R)$. The main result gives an estimate from below of the trace of a function from the Besov space on a uniformly distributed discrete subset. We also prove the corresponding result in the multivariate case and  discuss some applications to irregular approximate sampling in critical Besov spaces.
\end{abstract}

\maketitle

\section{Introduction}

\subsection{Motivation and main results}
Let us recall the classical Heisenberg uncertainty principle
\[
\int_{\R}|\zeta|^2|\hat{f}(\zeta)|^2\,\mbox{d}\zeta\int_{\R}|x|^2|f(x)|^2\,\mbox{d}x\ge \frac{\|f\|_{L^2}^4}{16\pi^2},\]
 where $f\in L^2(\R)$ and $\dst\hat{f}(\zeta)=\int_{\R}f(x)\exp(-2\pi i x\zeta)\,\mbox{d}x$.
It is frequently refereed to as a restriction for a simultaneous good localization of a function and its Fourier transform around the origin.  
It can also be read as an inequality preventing smooth functions with bounded norm (in some homogeneous Sobolev space) to be well concentrated, since
\[
\int_{\R}|\zeta|^2|\hat{f}(\zeta)|^2d\zeta=\|f'\|_{L^2}^2=\|f\|_{\dot{W}^{1,2}}^2.
\]
Here we use standard notation, $W^{s,p}(\R)=\{f\in L^p: (I-\Delta)^{s/2}f\in L^p\}$ for $s\ge 0$ and $p\ge 1$ and define the (homogeneous) norm in $W^{s,p}(\R)$ by $\|f\|_{\dot{W}^{s,p}}=\|\Delta^{s/2}f\|_{L^p}$ for $f\in W^{s,p}(\R)$.

The Heisenberg inequality  was further generalized by Cowling and Price:
\begin{fact}{\bf{(Cowling-Price, \cite{CP}).}}
{\sl Let $p,q\in[1,\infty]$ and $a,b>0$. There exists a constant $K$ such that
\begin{equation}
\label{eq:cp1}
\bigl\||x|^af\bigr\|_{L^p}+\bigl\||\xi|^b\widehat{f}\bigr\|_{L^q}\geq K\norm{f}_{L^2}
\end{equation}
for all $f\in L^2(\R)$ if and only if
$$
a>\frac{1}{2}-\frac{1}{p}\quad\mathrm{and}\quad b>\frac{1}{2}-\frac{1}{q}.
$$
If this is the case, let $\gamma$ be given by
$$
\gamma\left(a-\frac{1}{2}+\frac{1}{p}\right)=(1-\gamma)\left(b-\frac{1}{2}+\frac{1}{q}\right)
$$
then there exists a constant $K$ such that
$$
\bigl\||x|^af\bigr\|_{L^p}^\gamma\bigl\||\xi|^b\widehat{f}\bigr\|_{L^q}^{1-\gamma}\geq K\norm{f}_{L^2}.
$$}
\end{fact}
Again, when $q=2$, $\bigl\||\xi|^b\widehat{f}\bigr\|_{L^2(\R)}=\|f\|_{\dot{W}^{b,2}}$.

In \cite{str} R. Strichartz obtained a number of uncertainty inequalities in Euclidean spaces, connecting smoothness
of a function to its concentration on a suitably uniformly distributed set.
His starting point was the following:

\begin{fact}{\bf{(Strichartz, \cite{str}).}}
{\sl Let $\{a_j\}_{-\infty}^\infty$ be a sequence of points with $a_{j+1}-a_j\le b$, $\lim_{j\to\pm\infty} a_j=\pm \infty$, and let $f\in W^{1,2}(\R)$. If $\sum_j|f(a_j)|^2\le(1-\varepsilon)^2b^{-1}\|f\|_{L^2}^2$ then $\|f'\|_{L^2}^2\ge c\varepsilon^2b^{-2}\|f\|^2_{L^2}$.}
\end{fact}

Our aim here is to extend this result to $L^p$-smoothness setting as the Cowling-Price Theorem extends the classical Heisenberg Inequality. To be able to consider the values of a functions $f$ at some points, we assume that the function possesses some smoothness. For the classical case of $L^2$-norm we need at least derivative of order $1/2$. The embedding theorem suggests that we would require at least $1/p$-smoothness for $L^p$-norms. The right scale turns out to be that of Besov spaces. Our main result gives  bounds (also from below) on the trace of functions from Besov spaces on some uniformly distributed  union of subspaces. 
Note also that Paley-Wiener spaces are included in Besov spaces so that our results cover some results in sampling theory, though with non-optimal density. In particular, it implies the following statement.

\begin{theorem}\label{th:main}
Assume that  $m$ is an integer $0<m\le d$, and $1\le p<\infty$, let also $D$ be given. There exist  constants $\delta, C_1, C_2$ that depend only on $d,p,m, D$ such that  if $f\in B^{m/p}_{p,1}(\R^d)$,
$\|f\|_{\dot{B}^{m/p}_{p,1}}\le N\|f\|_p$, and $b^{m/p}<\delta N^{-1}$ then for any  
$G=\R^{m-d}\times E$, where $E\subset \R^m$ is a discrete set such that $\cup_{x\in E}[x-b,x+b]^m=\R^m$ and each point belongs to at most $D$ distinct sets in this union,
the following inequality holds
\[C_1\|f\|_{L^p}\le b^{m/p}\left(\int_G|f(x)|^pd\H^{d-m}(x)\right)^{1/p}\le C_2\|f\|_{L^p}.\] 
\end{theorem}

We remark that the right hand side inequality follows from localization and trace properties of the Besov spaces,
the aim is to prove the left hand side estimate.
However our argument gives both inequalities simultaneously and in particular provides a quite simple proof of the classical trace estimate. We consider the critical Besov spaces for which the trace inequality holds, the result does not hold in the spaces $B_{p,q}^s$ with $s<m/p$. The concentration principle says that a function satisfying the conditions above cannot have too many gaps, the precise statement is given by the left hand side inequality above. We also rewrite the result as sampling inequalities for functions  in the corresponding Besov spaces and apply it to give estimates for irregular approximate sampling. Similar one-dimensional approximation results for the case of regular samples were obtained recently in \cite{JOU}.

The remaining of the paper is organized as follows. We start by recalling the main facts on Besov spaces that we use. 
Section 2 is then devoted to the one-dimensional version of Theorem \ref{th:main} and some corollaries. In section 3 we prove the main result in higher dimensions  and apply it to sampling theory.

\subsection{Preliminaries on wavelets and Besov spaces}
We recall very briefly the basics of multiresolution wavelet analysis (for details
see for instance \cite{Daub}). For an arbitrary integer $N\geq 1$ one can construct  functions  $\psi^0 \in\cc^N(\R)$ (called the scaling function) and 	$\psi^1\in\cc^N(\R)$
(called the mother wavelet), with 
\begin{enumerate}
\renewcommand{\theenumi}{\roman{enumi}}
\item $\psi^0,\psi^1\in \cc^N(\R)$ and are real valued;

\item for $\ell=0,1$, $m=0,\ldots,N$ and $k\geq1$ there exists $C_k$ such that $|\partial^m\psi^\ell(x)|\leq C_k(1+|x|)^{-k}$;

\item the set of functions $\psi_{j,k}^1\,:x\to 2^{j/2}\psi^1(2^jx-k)$, $j,k\in\Z$ forms an orthonormal basis for $L^2(\R)$.

\item $\psi^1$ has $N$ vanishing moments $\dst\int_{\R}x^k\psi^1(x)\,\mbox{d}x=0$ for $k=0,\ldots,N-1$.
\end{enumerate}

We will then say that $\psi^1$ is $N$-regular.

Now, in higher dimension $d\geq 2$, we introduce
$$
\mathbf{0}^d=(0,\ldots,0),\quad \mathbf{1}^d=(1,\ldots,1)\quad\mbox{and}\quad
L^d=\{0,1\}^d\setminus\{\mathbf{0}^d\}.
$$
An orthonormal basis of $L^2(\R^d)$ is then obtained by tensorization:
for $\lambda=(\mathbf{k},\mathbf{l})\in\Z^d\times\{0,1\}^d$ with $\mathbf{k}=(k_1,\ldots,k_d)$
and $\mathbf{l}=(l_1,\ldots,l_d)$ we define
$$
\psi_{j, \lambda}(x)=\prod_{i=1}^d\psi^{l_j}_{j,k_i}(x).
$$
Any function $f\in L^2(\R^d)$ can then be written as
$$
f=\sum_{(j,\lambda)\in\Z\times\Z^d\times L^d}c_{j,\lambda}\psi_{j,\lambda}\quad\mbox{with }c_{j,\lambda}=\int_{\R^d}f(x)\psi_{j,\lambda}(x)\,\mbox{d}x.
$$

Let $s>0$, $1\leq p; q\leq\infty$. Assume that $\psi^1$ is at least $[s]+1$-regular. 
According to \cite{Bo,Ky,mey1}, the homogeneous Besov space $\dot{B}^s_{p,q}(\R^d)$ can be defined
as the space of all locally integrable functions such that the Besov norm
\begin{equation}
\label{eq:normbesov}
\norm{f}_{\dot{B}^s_{p,q}(\R^d)}=
\left(\sum_{j\in\Z}\left[2^{(s-d/p+d/2)j}\left(\sum_{\lambda\in\Z^d\times L^d}|c_{j,\lambda}|^p\right)^{\frac{1}{p}}
\right]^q\right)^{\frac{1}{q}}
\end{equation}
is finite.

An alternative definition is as follows \cite{Pet,Tri,Tbook}. Fix  an arbitrary non-negative smooth function $\rho$ supported in
$\{y\in\R^d\,:\dst\frac{1}{2}<|y|<2\}$ such that, for $y\not=0$,
$$
\sum_{j\in\Z}\rho(2^{-j}y)=1.
$$
Denote by $\ff$ the Fourier transform an $\ss'(\R^d)$ (the space of tempered distributions) and by $\ff^{-1}$ the inverse Fourier transform.
Then $\dot{B}^s_{p,q}(\R^d)$ is the space of all tempered distributions such that
\begin{equation}
\label{eq:normbesov2}
\left(\sum_{j\in\Z}2^{jsq}\norm{\ff^{-1}\bigl[\rho(2^{-j}\cdot)\ff[f]\bigr]}_{L^p(\R^d)}^q\right)^{\frac{1}{q}}.
\end{equation}
Moreover, this quantity defines an equivalent norm to $\norm{f}_{\dot{B}^s_{p,q}(\R^d)}$.
Using this norm, we see that there is a constant $C$ such that, if $\supp\ff(f)\cap\{|\xi|<b\}=\emptyset$
then
$\norm{f}_{\dot{B}^s_{p,q}(\R^d)}\geq Cb^{s-s'}\norm{f}_{\dot{B}^{s'}_{p,q}(\R^d)}$ if $s>s'$.
On the other hand, if we write $PW_b^p=\{f\in L^p\,: \supp\ff(f)\subset [-b,b]\}$ for the closed subspace of $L^p$
with distributional Fourier transform supported in $[-b,b]$, then $PW_b^p\subset B^s_{p,q}$ for every $s>0$ and $q\geq 1$.

Note that Besov spaces are also related to Sobolev spaces in the following way for $p\geq 1$ and $\eps>0$, $q'\geq q>0$
$$
W^{s+\eps,p}\hookrightarrow \dot{B}^s_{p,q}\hookrightarrow \dot{B}^s_{p,q'}\hookrightarrow W^{s-\eps,p}.
$$
Functions in the Besov space $B^{d/p}_{p,1}(\R^d)$ coincide with continuous ones almost everywhere. We will take the continuous representative.  The trace of a   function in $B_{p,1}^{d/p}(\R^d)$ on a subspace $M$ of codimension $r$ belongs to $B_{p,1}^{r/p}(M)$. We refer the reader to \cite{Tbook} and references there for the details; interesting results on local regularity of functions from the critical Besov spaces can be found in \cite{JM}.

\section{Sampling and uncertainty in dimension one}

\subsection{Localization inequalities}
First, we prove one-dimensional version of Theorem \ref{th:main} to demonstrate the main ideas avoiding technical complications. In the next section we outline the changes that should be done in multivariate case.
The proof is rather classical and shares some features with Stricharz's original approach but employs wavelet decomposition. For applications of similar techniques to irregular sampling see e.g. \cite{Gr1,Gr2}.

\begin{proposition}\label{pr:1}
Assume that $f\in B^{1/p}_{p,1}(\R)$, for some $p\ge 1$. If  $\{a_n\}_{n\in\Z}$ is an increasing sequence, $\lim_{\pm\infty}a_n=\pm\infty$, $a_{n+1}-a_n\le b$ and 
\[\sum_{n\in\Z}|f(a_n)|^p\le (1-\varepsilon)^pb^{-1}\|f\|_{L^p}^p,\] then $\|f\|_{\dot{B}^{1/p}_{p,1}}\ge c\varepsilon b^{-1/p}\|f\|_{L^p}$, where $c$ depends on $p$ only. \\
Moreover, there is a constant $C$, depending only on $p$, such that,
if $b\leq\dst C\eps\left(\frac{\|f\|_{L^p}}{\|f\|_{\dot{B}^{1/p}_{p,1}}}\right)^p$ and $a_j$ are as above with $b/2\le a_{j+1}-a_j\le b$ then
$$
\frac{1}{2b^{1/p}}\|f\|_{L^p}\leq\left(\sum_{n\in\Z}|f(a_n)|^p\right)^{1/p}
\leq\frac{5}{2b^{1/p}}\|f\|_{L^p}.
$$
\end{proposition}
 
\begin{proof} Let $\psi$ be a wavelet function such that $\psi(x)=0$ when $|x|>R$ ($\psi=\psi^1$ from the previous section). 
Write 
$$
f=\sum_{j\in\Z}\sum_{k\in\Z}\scal{f,\psi_{j,k}}\psi_{j,k},
$$
where $\psi_{j,k}$ are defined as above. It follows from the estimates below that the series converges uniformly. 
For each $n\in \Z$, let $I_n=[(a_{n-1}+a_{n})/2,(a_n+a_{n+1})/2]$, then $b_n=|I_n|\le b$. Fix some $n$ and  let $x\in I_n$. Then
\begin{eqnarray*}
f(x)-f(a_n)&=&\sum_{j\in\Z}\sum_{k\in\Z}\scal{f,\psi_{j,k}}\bigl(\psi_{j,k}(x)-\psi_{j,k}(a_n)\bigr)\\
&=&\sum_{j\in\Z}2^{j/2}\sum_{k\in\Z}\scal{f,\psi_{j,k}}\bigl(\psi_0(2^jx-k)-\psi_0(2^ja_n-k)\bigr).
\end{eqnarray*}
If $\psi(2^jx-k)\not=0$ then 
$$
2^jx-R\leq k\leq 2^jx+R.
$$
Write
\begin{gather*}
Z_x(j)=\{k\,: 2^jx-R\leq k\leq 2^jx+R\},\ Z_x(n,j)=Z_x(j)\cup Z_{a_n}(j),\ Z(n,j)=\bigcup_{x\in I_n}Z_x(j),
\end{gather*}
and note that $|Z_x(j)|\leq 2R+1$, therefore $|Z_x(n,j)|\leq 4R+2$.

Now, 
\begin{multline}\label{eq:f1}
|f(x)-f(a_n)|\leq\sum_{j\in\Z}2^{j/2}
\max_k|\psi(2^jx-k)-\psi(2^ja_n-k)|\sum_{k\in \Z_x(n,j)}|\scal{f, \psi_{j,k}}|\\
\le (4R+2)^{1/p'}\sum_{j\in\Z}2^{j/2}
\max_k|\psi(2^jx-k)-\psi(2^ja_n-k)|E_{n,j},
\end{multline}
the last inequality follows from H\"older's inequality with the notation $\dst\frac{1}{p}+\frac{1}{p'}=1$ and
\begin{equation}\label{eq:E}
E_{n,j}=\left(\sum_{k\in Z(n,j)}|\scal{f, \psi_{j,k}}|^p\right)^{1/p}.
\end{equation}

Further, we obtain
 $|Z(n,j)|\le 2R+1+2^jb$  and
if $|n-m|>\dst 2^{1-j}Rb^{-1}+1$ then $Z(n,j)\cap Z(m,j)=\emptyset$. Therefore, for each $j$, each $k\in\Z$
belongs to at most $\dst 2^{1-j}Rb^{-1}+2$ different $Z(n,j)$'s. 
We choose $j_0$ such that $2^{-j_0}\in (b,2b]$, and set $$M_j:=\sup_{k\in\Z}|\{n\,:k\in Z(n,j)\}|.$$
Then we get
\begin{equation}
M_j\leq 2^{1-j}Rb^{-1}+2 \leq
\begin{cases}C&\mbox{if }j\geq j_0\\
Cb^{-1}2^{-j}&\mbox{if }j< j_0
\end{cases},\label{eq:mj}
\end{equation}
where $C$ depends on $R$ only.

Clearly, for $x,a\in I_n$ we have
$$
|\psi(2^jx-k)-\psi(2^ja-k)|\le\begin{cases}2\|\psi\|_\infty,\\
2^j|x-a|\|\psi'\|_\infty.\end{cases}
$$
Then there exists a constant $C$ that depends only on $\psi$ such that
\begin{equation}\label{eq:fn}
|f(x)-f(a_n)|\le C\sum_{j\ge j_0}2^{j/2}E_{n,j}+C\sum_{j<j_0}2^{3j/2}|x-a_n|E_{n,j}. 
\end{equation}
Taking the $L^p$-norms over $x\in I_n$ and applying the triangle
 inequality, we get
\begin{eqnarray*}
\abs{\left(\int_{I_n}|f(x)|^p\,\mbox{d}x\right)^{1/p}-|f(a_n)|b_n^{1/p}}&\leq&
\left(\int_{I_n}|f(x)-f(a_n)|^p\,\mbox{d}x\right)^{1/p}\\
&\leq&C\sum_{j\ge j_0}2^{j/2}b^{1/p}E_{n,j}+C\sum_{j<j_0}2^{3j/2}b^{1+1/p}E_{n,j}.
\end{eqnarray*}

It remains to take the $\ell^p$ norm in $n$ to obtain
\begin{eqnarray*}
\abs{\norm{f}_{L^p(\R)}-\left(\sum_{n\in\Z}b_n|f(a_n)|^p\right)^{1/p}}&\leq&
\left(\sum_{n\in\Z}\ent{\int_{I_n}|f(x)|^p\,\mbox{d}x-|f(a_n)|^pb_n}\right)^{1/p}\\
&\leq&C\sum_{j\geq j_0}2^{j/2}\left(\sum_{n\in\Z}\sum_{k\in Z(n,j)}|\scal{f,\psi_{j,k}}|^p\right)^{1/p}b^{1/p}\\
&&+C\sum_{j< j_0}2^{3j/2}\left(\sum_{n\in\Z}\sum_{k\in Z(n,j)}|\scal{f,\psi_{j,k}}|^p\right)^{1/p}b^{1+1/p}.
\end{eqnarray*}
We have also
$$
\sum_{n\in\Z}\sum_{k\in Z(n,j)}|\scal{f,\psi_{j,k}}|^p\leq M_j\sum_{k\in\Z}|\scal{f,\psi_{j,k}}|^p,
$$
and applying \eqref{eq:mj}, we get
\begin{eqnarray*}
\abs{\norm{f}_{L^p(\R)}-\left(\sum_{n\in\Z}b_n|f(a_n)|^p\right)^{1/p}}&\leq&
C\sum_{j\geq j_0}2^{j/2}\left(\sum_{k\in\Z}|\scal{f,\psi_{j,k}}|^p\right)^{1/p}b^{1/p}\\
&&+C\sum_{j<j_0}2^{(2-1/p)j}\left(\sum_{k\in\Z}|\scal{f,\psi_{j,k}}|^p\right)^{1/p}b.
\end{eqnarray*}
Finally,  notice that if $j\leq j_0$ then $2^{(1-1/p)j}\leq Cb^{1/p-1}$, hence
\begin{eqnarray}\label{eq:intB}
\abs{\norm{f}_{L^p(\R)}-\left(\sum_{n\in\Z}b_n|f(a_n)|^p\right)^{1/p}}
&\leq&C\sum_{j\in\Z}2^{j/2}\left(\sum_{k\in\Z}|\scal{f,\psi_{j,k}}|^p\right)^{1/p}b^{1/p}\\
&=&C\|f\|_{\dot{B}^{1/p}_{p,1}}b^{1/p}\nonumber
\end{eqnarray}
since $\dst\frac{1}{2}=s-\frac{1}{p}+\frac{1}{2}$ when $s=1/p$.
The inequality $\|f\|_{\dot{B}^{1/p}_{p,1}}\geq cb^{-1/p}\varepsilon$ follows since $b_n\le b$.

Further, if $C\|f\|_{\dot{B}^{1+1/p}_{p,1}}b^{1/p}\leq 1/2\|f\|_{L^p(\R)}$  and $b_n=(a_{n+1}-a_{n-1})/2\ge b/2$ then
$$
\frac{1}{2b^{1/p}}\|f\|_{L^p(\R)}\leq\left(\sum_{n\in\Z}|f(a_n)|^p\right)^{1/p}
\leq\frac{5}{2b^{1/p}}\|f\|_{L^p(\R)}.
$$
\end{proof}

\subsection{Some corollaries}
As a first application of this proposition, let us establish the following version of the Uncertainty Principle:

\begin{corollary} Let $p\geq 1$ and $\alpha>0$. Let $f\in L^p$ then there exists $c$ that depends only on $p$ and $\alpha$ such that 
\[\||x|^{\alpha/p}f\|_{L^p(\R)}\|f\|_{\dot{B}^{1/p}_{p,1}}^\alpha\ge c_p\|f\|^{1+\alpha}_{L^p(\R)}.\]
\end{corollary}

Note that a stronger version has been recently obtained  by J. Martin and M. Milman in \cite{MM}, see the lines following inequality (5.13) in \cite{MM}.

\begin{proof}
We may assume that $\|f\|_p=1$. Suppose that $\||x|^{\alpha/p}f\|_p=A$ then for any $b>0$
\[
\int_{|x|>b/4}|f(x)|^p\,\mbox{d}x\le (4b^{-1})^\alpha\int_{|x|>b/4}|x|^\alpha|f(x)|^p\mbox{d}x\le 4^\alpha b^{-\alpha}A^p.
\]
But
$$
\int_{|x|>b/4}|f(x)|^p\,\mbox{d}x=\int_{-b/4}^{b/4}\sum_{j\in\Z\setminus\{0\}}|f(x+jb/2)|^p\,\mbox{d}x,
$$
therefore there exists $a\in (-b/4, b/4)$ such that $\sum_{j\in\Z\setminus\{0\}}|f(a+jb/2)|^p\le 4^\alpha b^{-1-\alpha}A^p$.

We choose $b$ such that $4^\alpha b^{-\alpha}A^p=1/2$. Then by the proposition above we obtain 
\[
\|f\|_{\dot{B}^{1/p}_{p,1}}\ge Cb^{-1/p}=CA^{-1/\alpha}.
\]
This implies the required inequality.
\end{proof}

We will now show some sampling estimates for functions in certain Besov spaces:

\begin{corollary} Let $f\in B^{1/p}_{p,1}(\R)$ and $a=\{a_j\}$ be a sequence as above.
\begin{enumerate}
\renewcommand{\theenumi}{\roman{enumi}}
\item If $S_1(f,a)$ is the piece-wise linear interpolant of $f$ at $a$. Then 
\[
\|f-S_1(f,a)\|_{L^p}\le Cb^{1/p}\|f\|_{\dot{B}^{1/p}_{p,1}}.
\]
\item There exists a bandlimited function $g\in[-cb^{-1},cb^{-1}]$ such that
\[ \|f-g\|_{L^p}\le Cb^{1/p}\|f\|_{\dot{B}^{1/p}_{p,1}}\quad {\text{and}}\quad \|f(a_j)-g(a_j)\|_{\ell^p}\le C\|f\|_{\dot{B}^{1/p}_{p,1}}.\]
If in addition $f\in B^{s}_{p,\infty}$ for some $s>1/p$ then 
\[\|f-g\|_{L^p}\le C_sb^{s}\|f\|_{\dot{B}^{s}_{p,\infty}}\quad {\text{and}}\quad \|f(a_j)-g(a_j)\|_{\ell^p}\le C_sb^{s-1/p}\|f\|_{\dot{B}^{s}_{p,\infty}}.\] 
\end{enumerate}
\end{corollary}

\begin{proof} The first statement follows directly from the proof of the theorem. Clearly, 
\[|f(x)-S_1(f,a)(x)|\le\max\{|f(x)-f(a_{n+1})|, |f(x)-f(a_n)|\}\quad\text{for}\quad x\in[a_n,a_{n+1}].\] 
Thus integrating \eqref{eq:f1} we obtain the required inequality.

To prove the second statement, let us now assume that $\psi_1$ is regular and band-limited to some interval $[-\Omega,\Omega]$.
Let $f\in \dot{B}^{1/p}_{p,q}$ and write
$$
f=\sum_{j\in\Z}\sum_{k\in\Z}\scal{f,\psi_{j,k}}\psi_{j,k}.
$$
Let $j_0$ be such that $2^{j_0}\leq b^{-1}\leq 2^{j_0+1}$ and define
$$
g=\sum_{j\leq j_0}\sum_{k\in\Z}\scal{f,\psi_{j,k}}\psi_{j,k}\quad\mbox{and}\quad h=f-g.
$$
Then $
\|h\|_{\dot{B}^{1/p}_{p,q}}\le C\|f\|_{\dot{B}^{1/p}_{p,q}}. $
Further we have
\begin{eqnarray*}
 \|h\|_{L^p}&\leq&C\left(\sum_{j>j_0}2^{(-1+p/2)j}\sum_{k}|\scal{f,\psi_{j,k}}|^p\right)^{1/p}\\
&\leq&C2^{-j_0/p}\left(\sum_{j>j_1}2^{pj/2}\sum_{k}|\scal{f,\psi_{j,k}}|^p\right)^{1/p}\\
&\leq&Cb^{1/p}\sum_{j>j_0}2^j\left(\sum_{k}|\scal{f,\psi_{j,k}}|^p\right)^{1/p}
\end{eqnarray*}
and thus $ \|h\|_{L^p}\leq Cb^{1/p}\|h\|_{\dot{B}^{1/p}_{p,q}}\le Cb^{1/p}\|f\|_{\dot{B}^{1/p}_{p,q}}$.

Moreover, if we assume that $f\in B^{s}_{p,\infty}$ for some $s>1/p$ then 
\[\|h\|_{\dot{B}^{1/p}_{p,1}}\le C_sb^{s-1/p}\|h\|_{\dot{B}^{s}_{p,\infty}}\le C_sb^{s-1/p}\|f\|_{\dot{B}^{s}_{p,\infty}}\quad{\text{and}}\quad
\|h\|_{L^p}\le Cb^{s}\|f\|_{\dot{B}^{s}_{p,\infty}}.\] 
Now, we apply \eqref{eq:intB} to $h$ and obtain 
\[
\left(\sum_{n\in\Z}|h(a_n)|^p\right)^{1/p}\le C\left(b^{-1/p}\|h\|_{L^p}+\|h\|^p_{\dot{B}^{1/p}_{p,1}}\right). 
\]
The required inequalities follow.
\end{proof}

Some results on regular smooth spline interpolation in such Besov spaces were obtained
in \cite{JOU}, we discuss a different sampling for the multivariate case  in the next section.

\section{Multivariate concentration inequality and  irregular sampling}

\subsection{Localization inequalities in higher dimensions}
Similar argument as in the previous section gives the sampling inequalities of Theorem \ref{th:main} formulated in the introduction for functions with small norms in appropriate homogeneous Besov spaces. Example of spaces of functions with homogeneous Besov norms controlled by the corresponding Lebesgue norm are bandlimited functions or more generally functions  in shift-invariant subspaces,  see \cite{A,AG}. Sampling of bandlimited functions from irregular point sets is a well developed topic, see for example \cite{FG1,FG,Gr1,Gr2,LM,Mad,Pes}; new interesting results on sampling from trajectories can be found in \cite{GRUV}. We obtain the sampling inequalities under  much milder smoothness assumptions and reduce the questions of approximate sampling of functions in Besov spaces to those of bandlimited functions.

Let us now describe the general setting of our sampling sets. Our aim here is to provide a description of sampling
sets that are intuitive and easy to check rather than a fully general definition that would be too complicated
to check in practice.
Let $1\leq m\leq d$ be an integer and $b,C_0>0$ be real. 

First we take $G$ to be a finite or countable union of $d-m$ dimensional $\cc^1$ manifolds (a countable set when $m=d$).
To each $a\in G$ we attach an $m$-dimensional manifold $H_a$ in a sufficiently regular way ({\it e.g.}
$H_a$ may be defined through an implicit function). Each $H_a$ is endowed with a measure $\nu_a$ that is absolutely continuous with respect to the corresponding Hausdorff (surface) measure on $H_a$,  $\nu_a=\ffi\mbox{d}\H^{m}\Big|_{H_a}$ with $C_0^{-1}\leq\ffi\leq C_0$, measures $\nu_a$ depend on $a$ also in a regular way, see (ii) below. 
 We further assume that
\begin{enumerate}
\renewcommand{\theenumi}{\roman{enumi}}
\item for any $a$ the diameter of $H_a$ is bounded in the following way $b/2\leq\diam(H_a)\leq b$,

\item for every measurable set $E\subset \R^d$ sets $E\cap H_a$ are $\nu_a$-measurable for $\H^{m-d}$-almost all $a\in G$ and $C_0^{-1}\lambda_d(E)\le\int_G \nu_a(E\cap H_a)d\H^{m-d}(a)\le C_0\lambda_d(E)$; 
this implies that
for every 
$f\in L^1(\R^d)$,
\begin{equation}\label{eq:equiv}
\frac{1}{C_0}\int_{\R^d}|f(x)|\,\mbox{d}x\leq
\int_G\int_{H_a}|f(x)|\,\mbox{d}\nu_a(x)\,\mbox{d}\H^{d-m}(a)\leq
C_0\int_{\R^d}|f(x)|\,\mbox{d}x,
\end{equation}

\item for every $x\in\R^d$, $R>0$,
\begin{equation}\label{eq:mes2}
C_0^{-1}\min(R,b)^m\leq \nu_a\bigl(B(x,R)\cap H_a\bigr)\leq C_0\min(R,b)^m,
\end{equation}

\item for every $x\in\R^d$, $R>0$
\begin{equation}\label{eq:mes}
\H^{d-m}(G\cap B(x,R))\le C_0 R^{d-m}\max\{1, Rb^{-1}\}^m.
\end{equation}

\end{enumerate}

\begin{examplenum}\rm  (i) Let $(a_n)_{n\in\Z}$ be an increasing sequence such that 
$\frac{b}{2}\leq a_{n+1}-a_n\leq b$
then we can take $G=\cup_{n\in\Z}P_n$, where $P_n=\R^{d-1}\times\{a_n\}$, here $m=1$. 
For $a=(\alpha,a_n)\in G$ we take $\dst H_a=\{\alpha\}\times\ent{\frac{a_{n-1}+a_n}{2},\frac{a_{n+1}+a_n}{2}}$.

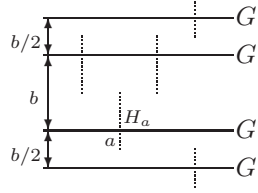
\begin{figure}[!ht]
\begin{center}
\setlength{\unitlength}{0.5cm}
\begin{picture}(7,5)
\dottedline{0.02}(1,0.5)(6,0.5)\put(6.1,0.3){$G$}
\dottedline{0.02}(1,1.5)(6,1.5)\put(6.1,1.3){$G$}
\dottedline{0.02}(1,3.5)(6,3.5)\put(6.1,3.3){$G$}
\dottedline{0.02}(1,4.5)(6,4.5)\put(6.1,4.3){$G$}

\put(1.1,0.5){\vector(0,1){1}}
\put(1.1,1.5){\vector(0,-1){1}}
\put(1.1,1.5){\vector(0,1){2}}
\put(1.1,3.5){\vector(0,-1){2}}
\put(1.1,3.5){\vector(0,1){1}}
\put(1.1,4.5){\vector(0,-1){1}}

\put(0.1,0.7){$\scriptstyle{b/2}$}\put(0.6,2.2){$\scriptstyle{b}$}\put(0.1,3.7){$\scriptstyle{b/2}$}
\put(2.6,1.1){$\scriptstyle{a}$}\dottedline{0.1}(3,1)(3,2.5)\put(3.1,1.7){$\scriptstyle{H_a}$}
\dottedline{0.1}(2,4)(2,2.5)\dottedline{0.1}(4,4)(4,2.5)\dottedline{0.1}(5,0)(5,1)\dottedline{0.1}(5,4)(5,5)
\end{picture}
\caption{ Example (i) in dimension $2$, $G=\bigcup\R\times\{a_n\}$ }
\label{fig:cor}
\end{center}
\end{figure}

(ii) The hyperplanes $P_n$ can be replaced by more general manifolds. For instance,
let $(a_n)_{n\in\Z}$ be a sequences such that $\frac{b}{2}\leq a_{n+1}-a_n\leq b$ and let
$f\,:\R^{d-1}\to\R$ be smooth bounded function with bounded derivative. Let $\tilde{P}_n=\bigl\{\bigl(x',f(x')+a_n\bigr),x'\in\R^{d-1}\bigr\}$
then we can take $G=\cup_{n\in\Z}\tilde{P}_n$. Then if $a=\bigl(x',f(x')+a_n\bigr)$ we can again set
$H_a=\{a+(0,t)\,:|t|\leq b+\max|f|\}$.

(iii) For each $k\in\Z^{d-1}$, let $f_{k}\,:\R\to\R^{d-1}$ be a smooth function such that $f_{k}$ takes its values in the 
$\ell^\infty$-ball centered at $bk$ of radius $b/4$ and such that the derivatives are uniformly bounded from above and below
$C_0^{-1}\leq \|f_k^\prime\|\leq C_0$. Let $P_k=\bigl\{\bigl(f_k(x),x\bigr),x\in\R\bigr\}\subset\R^d$ and $G=\bigcup_{k\in\Z^{d-1}}P_k$.
To each $\bigl(f_k(a),a\bigr)\in G$ we associate $H_a$ to be the $\ell^\infty$-ball centered at $(bk,a)$ of radius $b/4$ endowed with the Lebesgue measure.

\begin{figure}[!ht]
\begin{center}
\setlength{\unitlength}{0.5cm}
\begin{picture}(12,12)

\dottedline{0.02}(1,1)(5,1)\dottedline{0.2}(9,1)(5,1)\dottedline{0.2}(13,5)(9,5)
\dottedline{0.02}(1,1)(3,3)\dottedline{0.2}(9,1)(11,3)\dottedline{0.2}(13,5)(11,3)
\dottedline{0.02}(7,3)(3,3)\dottedline{0.2}(7,3)(11,3)
\dottedline{0.02}(7,3)(5,1)\dottedline{0.2}(7,3)(5,1)\dottedline{0.2}(7,3)(9,5)

\put(1,0.8){\vector(1,0){4}}\put(5,0.8){\vector(-1,0){4}}\put(3,0.1){$\scriptstyle{b}$}
\put(1,1.5){\vector(1,1){1}}\put(2,2.5){\vector(-1,-1){1}}
\put(0,2){$\scriptstyle{b/2}$}
\dottedline{0.2}(1,6)(5,6)
\dottedline{0.2}(1,6)(3,8)
\dottedline{0.2}(7,8)(3,8)
\dottedline{0.2}(7,8)(5,6)
\put(5.2,7.5){$\scriptstyle{H_a}$}
\put(3.7,7){$\scriptstyle{a}$}

\dottedline{0.1}(2.5,1.5)(4.5,1.5)
\dottedline{0.1}(2.5,1.5)(3.5,2.5)
\dottedline{0.1}(4.5,1.5)(5.5,2.5)
\dottedline{0.1}(3.5,2.5)(5.5,2.5)

\qbezier[200](4,2)(5,4)(4,6)\qbezier[200](4,6)(3,8)(4,10)\put(4.1,9.5){$G$}
\qbezier[200](8,2)(7,4)(8,6)\qbezier[200](8,6)(9,8)(8,10)\put(8.2,9.8){$G$}
\qbezier[200](10,4)(9.5,6)(10,8)\qbezier[200](10,8)(10.5,10)(10,12)\put(10.1,11.7){$G$}
\end{picture}
\caption{Example (iii) in dimension 3}
\label{fig:cor}
\end{center}
\end{figure}
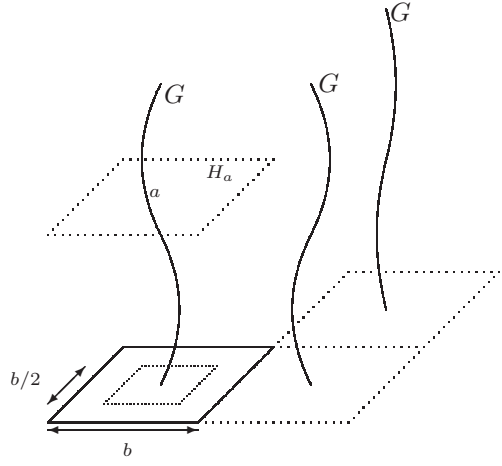

\end{examplenum}

We will prove the following generalization of Theorem 1.1

\begin{theorem} Suppose that $1\le p\le \infty$ and $G$ as above.
There exist  constants $\delta, C_1, C_2$ that depend only on $d,p,m,D, C_0$  such that  if $f\in B^{m/p}_{p,1}(\R^d)$,
$\|f\|_{\dot{B}^{m/p}_{p,1}}\le N\|f\|_p$, and $b^{m/p}<\delta N^{-1}$ then
\[C_1\|f\|_{L^p}\le b^{m/p}\left(\int_G|f(x)|^pd\H^{d-m}(x)\right)^{1/p}\le C_2\|f\|_{L^p}.\] 
\end{theorem}

\begin{proof}	
 The proof repeats the one-dimensional argument from the previous section. We start with a smooth one-dimensional wavelet function supported on $[-R,R]$ and construct a wavelet basis as outlined in the introduction. We have
$$
f=\sum_{j\in\Z}\sum_{\lambda\in\Z^d\times L^d}\scal{f,\psi_{j,\lambda}}\psi_{j,\lambda}.
$$ 
Fix some $a$ and  let $x\in H_a$. Then
\begin{eqnarray*}
f(x)-f(a)=\sum_{j\in\Z}2^{dj/2}\sum_{\lambda=(k,l)\in\Z^d\times L^d}\scal{f,\psi_{j,\lambda}}\bigl(\psi_{0,l}(2^jx-k)-\psi_{0,l}(2^ja-k)\bigr).
\end{eqnarray*}

Write
\begin{gather*}
Z_x(j)=\{k\in\Z^d: 2^jx-k\in[-R,R]^d\},\ Z_x(a,j)=Z_x(j)\cup Z_{a}(j),\\ Z(a,j)=\cup_{x\in H_a}Z_x(j),
\end{gather*}
and note that, $|Z_x(j)|\leq C$, where $C$ depends on $R$ only.
We also define
$$
E_{a,j}=\left(\sum_{\lambda\in Z(a,j)\times L^d}|\scal{f, \psi_{j,\lambda}}|^p\right)^{1/p}.
$$
Further, it is not difficult to see that
 $|Z(a,j)|\le C(1+2^{jm}b^{m})$ and by \eqref{eq:mes} the following inequality holds
$$
M_j:=\sup_{k\in\Z^d}\H^{d-m}\{a\in G\,:k\in Z(a,j)\}|\leq \begin{cases}C_0 2^{-j(d-m)}&\mbox{if }j\geq j_0\\
C_0b^{-m}2^{-dj}&\mbox{if }j< j_0
\end{cases},
$$
where $2^{-j_0}\in (b,2b]$ as before.
Then \eqref{eq:fn} becomes
$$
|f(x)- f(a)|\le C\sum_{j\ge j_0}2^{dj/2}E_{a,j}+C\sum_{j<j_0}2^{(d+2)j/2}|x-a|E_{a,j}. 
$$
We take the $L^p(H_a,\nu_a)$-norms over $H_a$. Then we apply the triangle inequality and get
\begin{multline*}
\left|\left(\int_{H_a}|f(x)|^p d\nu_a\right)^{1/p}-
|f(a)|\nu_a(H_a)^{1/p}\right|\\
\le
C\sum_{j\ge j_0}2^{dj/2}b^{m/p}E_{a,j}+C\sum_{j<j_0}2^{(d+2)j/2}b^{1+m/p}E_{a,j}.
\end{multline*}
Next, taking the $L^p$ norm over each $G$, we have
\begin{multline}\label{eq:new}
\left|\left(\int_G\int_{H_a}|f(x)|^p\,\mbox{d}\nu_a(x)\,\mbox{d}\H^{d-m}(a)\right)^{1/p}-
\|f(a)\nu_a(H_a)^{1/p}\|_{L^p(G)}\right|\\
\le
C\sum_{j\ge j_0}2^{dj/2}b^{m/p}\|E_{a,j}\|_{L^p(G)}+C\sum_{j<j_0}2^{(d+2)j/2}b^{1+m/p}\|E_{a,j}\|_{L^p(G)}.
\end{multline}
Futher, when $j\ge j_0$ the estimate for $M_j$ above implies
\[
\|E_{a,j}\|_{L^p(G)}\le C_02^{-j(d-m)/p}\left(\sum_{\lambda=(k,l)}|\scal{f, \psi_{j,\lambda}}|^p\right)^{1/p}.\]
Then the first sum is the right hand side of \eqref{eq:new} is  bounded by $Cb^{m/p}\|f\|_{\dot{B}_{p,1}^{m/p}}$. 

For $j<j_0$, we have
\[
\|E_{a,j}\|_{L^p(G)}\le C_0b^{-m/p}2^{-jd/p}\left(\sum_{\lambda=(k,l)}|\scal{f, \psi_{j,\lambda}}|^p\right)^{1/p}.\]
We divide the second sum in the right hand side of \eqref{eq:new} into two sums, where $j_1<j_0$ will be fixed later,
\begin{multline*}
C\sum_{j<j_0}2^{(d+2)j/2}b^{1+m/p}\|E_{a,j}\|_{L^p(G)}
\le C\left(\sum_{j<j_1}+\sum_{j=j_1}^{j_0}\right)b2^{j(d/2+1-d/p)}\left(\sum_{\lambda=(k,l)}|\scal{f, \psi_{j,\lambda}}|^p\right)^{1/p}\\
\le Cb\sum_{j<j_1}2^j\|f\|_{L^p(\R^d)}+Cb\max_{j_1\le j<j_0}\{2^{j(1-m/p)}\}\|f\|_{\dot{B}^{m/p}_{p,1}}\\
\le 
Cb2^{j_1}\|f\|_{L^p(\R^d)}+Cb\left(2^{j_0(1-m/p)}+2^{j_1(1-m/p)}\right)\|f\|_{\dot{B}^{m/p}_{p,1}},
\end{multline*}
where the constants depend on $\psi$ 
and does not depend on $j$ (it follows by a simple scaling argument). Since $b2^{j_0}\in[1/2,1]$, we obtain
\begin{multline*}
\sum_{j<j_0}2^{(d+2)j/2}b^{1+m/p}\|E_{a,j}\|_{L^p(G)}\le\\
C2^{j_1-j_0}\|f\|_{L^p(\R^d)}+Cb^{m/p}2^{m/p}\left(1+2^{(j_1-j_0)(1-m/p)}\right)\|f\|_{\dot{B}^{m/p}_{p,1}}.
\end{multline*}
In summary, \eqref{eq:new} implies
\begin{multline}
\label{eq:last}
\left|\left(\int_G\int_{H_a}|f(x)|^p\,\mbox{d}\nu_a(x)\,\mbox{d}\mu(a)\right)^{1/p}-
\|f(a)\nu_a(H_a)^{1/p}\|_{L^p(G)}\right|\\
\le C 2^{j_1-j_0}\|f\|_{L^p(\R^d)}+Cb^{m/p}C_{m,p}
\|f\|_{\dot{B}_{p,1}^{m/p}},
\end{multline}
where $C_{m,p}=C\Bigl(1+2^{m/p}\bigl(1+2^{(j_1-j_0)(1-m/p)}\bigr)\Bigr)$.
Taking the lower bound in \eqref{eq:equiv} and the upper bound in \eqref{eq:mes2} we get
$$
C_0^{-1}\norm{f}_{L^p(\R^d)}-C_0^{1/p}b^{m/p}\|f\|_{L^p(G)}\le C 2^{j_1-j_0}\|f\|_{L^p}+Cb^{m/p}C_{m,p}
\|f\|_{\dot{B}_{p,1}^{m/p}}.
$$
Choosing $j_1$ small enough to have $C 2^{j_1-j_0}\leq (2C_0)^{-1}$  we get
$$
(2C_0)^{-1}\norm{f}_{L^p(\R^d)}-C_{m,p}b^{m/p}\|f\|_{\dot{B}_{p,1}^{m/p}}\leq C_0^{1/p}b^{m/p}\|f\|_{L^p(G)}.
$$

On the other hand, taking the upper bound in \eqref{eq:equiv} and the lower bound in \eqref{eq:mes2}, \eqref{eq:last} with $j_1=j_0$
implies
$$
C_0^{-1/p}b^{m/p}\|f\|_{L^p(G)}-C_0\norm{f}_{L^p(\R^d)}\le C \|f\|_p+C(1+2^{m/p+1})b^{m/p}\|f\|_{\dot{B}_{p,1}^{m/p}}.
$$

The theorem follows immediately.
\end{proof}

\begin{remark}
The result may be extended to sampling sets that do not exactly satisfy the requirements.
For instance let $\{r_n\}_{n\in\N}$ be an increasing sequence of positive numbers, $r_{n+1}-r_n\in(b/2,b)$ and let 
$G=\cup_n\{x:|x|=r_n\}$ be the union of concentric spheres of radius $r_n$. For $a\in G$, $|a|=r_n$,
we define 
$$
H_a=\begin{cases}
\left\{ta,\  
\frac{r_{n-1}+r_n}{2}\leq t\leq\frac{r_{n+1}+r_n}{2}\right\}&\mbox{for }n\geq 1\\
\left\{ta,\  
0\leq t\leq\frac{r_0+r_1}{2}\right\}&\mbox{for }n=0
\end{cases}.
$$
The right hand side of \eqref{eq:equiv} does not hold and we have to replace it by
\begin{equation}\label{eq:equiv2}
\int_G\int_{H_a}|f(x)|\,\mbox{d}\nu_a(x)\,\mbox{d}\H^{d-m}(a)\leq
C_0\int_0^\infty\int_{\S^{d-1}}|f(r\zeta)|\,\mbox{d}\sigma(\zeta) \max(1,r^{d-1})\,\mbox{d}r
\end{equation}
while the left hand side of \eqref{eq:equiv} still holds.

Then the proof of the theorem shows that
$$
C_1\|f\|_{L^p(\R^d)}\leq
b^{m/p}\left(\int_G|f(x)|^pd\H^{d-m}(x)\right)^{1/p}\le C_2\|f\|_{L^p(\R^d,\nu)}
$$
where $\nu$ is the measure $\,\mbox{d}\sigma(\zeta)\max(1,r^{d-1})\,\mbox{d}r$ (in polar coordinates).
We then apply the left hand side of this inequality to $f\alpha$ where $\alpha$ is a smooth function such that $\alpha=0$ in a ball $B(0,\eps)$
and $\alpha=1$ outside the ball $B(0,1/4)$. Then $\norm{\alpha f}_{B_{p,1}^{m/p}}\leq C\norm{f}_{B_{p,1}^{m/p}}$
so that the theorem applies if $b$ is small enough. It remains to notice that
\begin{eqnarray*}
\left(\int_G|f(x)|^pd\H^{d-m}(x)\right)^{1/p}&=&
\left(\int_G|\alpha(x)f(x)|^pd\H^{d-m}(x)\right)^{1/p}\\
&\le& C_2\|\psi f\|_{L^p(\R^d,\nu)}\leq C\| \alpha f\|_{L^p(\R^d)}\leq C\|f\|_{L^p(\R^d)}.
\end{eqnarray*}
Therefore the theorem also holds in this case.

A similar reasoning also applies to a spiral. Let $r_k$ be as previously and let $\rho$ be a smooth (strictly)
increasing function such that $\rho(2k\pi)=r_k$. Let $G\subset\R^2$ be the curve given in polar coordinates by $\rho$,
that is $G=\{\rho(\theta)(\cos\theta,\sin\theta),\theta\in[0,+\infty)\}$. If $\theta\in[2k\pi,2(k+1)\pi]$
we attach to $a=\rho(\theta)(\cos\theta,\sin\theta)$ the manifold $H_a=t(\cos\theta,\sin\theta)$,
$r_{k-1}\leq t\leq r_{k+1}$ (with the convention $r_{-1}=0$).
\end{remark}

\subsection{Irregular sampling}
Let us now outline how the above results fit into 
existing general sampling procedures.
Suppose that $f\in B^{m/p}_{p,1}(\R^d)$ and $G$ is the union of $d-m$-dimensional subspaces as in the theorem, then we have the trace operator bounded in the following way
\[
T_G:B^{m/p}_{p,1}(\R^d)\to L^p(G),\quad \|T_Gf\|_{L^p(G)}\le C\left(b^{-m/p}\|f\|_{L^p(\R^d)}+\|f\|_{B^{m/p}_{p,1}(\R^d)}\right).
\]
Assume further that we are given a bounded operator $S_G: L^p(G)\to L^p(\R^d)$ that interpolates band-limited functions. More precisely,
we assume that there is an $A>0$ such that
\[
\|S_Gu\|_{L^p(\R^d)}\le Ab^{m/p}\|u\|_{L^p(G)}
\] 
and if $g$ is bandlimited with $\hat{g}(x)=0$ when $x\not\in[-cb^{-1},cb^{-1}]^d$ then
\[
\|g-S_G(T_Gg)\|_{L^p(\R^d)}\le B b^{m/p}\|g\|_{\dot{B}^{m/p}_{p,1}}.
\]    

Now, let $f\in B^{m/p}_{p,1}$. Fix a smooth bounded multiplier $\chi$,
such that $\supp(\chi)\subset[-cb,cb]^d$ and $\chi=1$ on $[-ab^{-1},ab^{-1}]^d$. 
Write $f=g+h$, where $g=P_\chi f=\mathcal{F}^{-1}(\hat{f}\chi)$ so that $\supp(g)\subset[-cb,cb]^d$ and $\hat{h}=0$ on $[-ab^{-1},ab^{-1}]^d$.
Then 
\begin{eqnarray*}
\|f-S_G(T_Gf)\|_{L^p(\R^d)}&\le&\|h\|_{L^p}+\|g-S_G(T_Gg)\|_{L^p(\R^d)}+\|S_G(T_Gh)\|_{L^p(\R^d)}\\
&\leq&\|g-S_G(T_Gg)\|_{L^p(\R^d)}+\|h\|_{L^p(\R^d)}+Ab^{m/p}\|h\|_{L^p(G)}.
\end{eqnarray*} 
Applying the theorem we see that the last two terms are bounded by $C(1+A)b^{m/p}\|f\|_{B^{m/p}_{p,1}}$.
If we further assume better smoothness for $f$, $f\in B^{s}_{p,\infty}$ with $s>1/p$, then they are even bounded 
by $C_s(1+A)b^{s}\|f\|_{B^{s}_{p,\infty}}$. 
Thus sampling of functions in Besov spaces can be reduced to sampling of bandlimited functions. 
This was done in \cite{JOU} for the case of  dimension one and regular samples.
We allow irregular sample sets and to claim the correct order of converges of reconstructions, 
\[
\|f-S_G(T_Gf)\|_p\le Cb^{m/p}\|f\|_{B^{m/p}_{p,1}},
\]
we need the constants $A$ and $B$ be uniform, i.e., they may depend only on $p,m,d$ and $D$ but not an the specific geometry 
of  the set $G$.

 It can be checked that for example the (iterative) sampling algorithm provided in \cite{A} can be applied, 
where the $L^2$-estimates can be replaced by $L^p$-estimates, related inequalities can be found in \cite{FG, Pes1}. 
Let $G=\Gamma\times\R^m$ as above, and let $\Lambda=(b\Z)^{m}$. We consider $\Lambda_G=\Gamma\times\Lambda$, it is a discrete separated set in $\R^d$. The usual averaging operator $V:L^p(G)\to l^p(\Lambda_G)$ is bounded, $\|V(u)\|_{l^p(\Lambda_G)}\le b^{(m-d)/p} \|u\|_{L^p(G)}$. Let $\{\beta_j\}$ be a smooth  bounded partition of unity adapted to $\{B(x,2b)\}_{x\in\Lambda_g}$ (see for example Definition 4.2 in \cite{A}). We define $A:l^p(\Lambda_G)\to L^p$ by $Ac=\sum_j c_j\beta_j$ and $A_1=AVT_G$. Then by Lemma 4.1 in \cite{FG1} choosing $c$ small enough we can guarantee that $I-P_\chi A_1$ is a contraction on $L^p\cap \mathcal{F}^{-1}(L^2([-cb,cb]^d)$, where $P_\chi:L^p\to L^p$ is defined by $P_\chi f=(\mathcal{F})^{-1}(\chi \hat{f})$ as above. Then there exists $N$ depending only on $p,m,d, c$ such that 
\[S_G=\sum_{k=0}^N(I-PA_1)^k P_\chi AV_G\] satisfies the required uniform estimates.

\end{document}